\documentclass[a4paper,reqno]{amsart}

\usepackage[utf8]{inputenc}
\usepackage[english]{babel}
\usepackage{amsmath,amssymb,amsfonts,amsthm}
\usepackage{microtype}
\usepackage{enumitem}
\usepackage{cite}
\usepackage{url}
\usepackage{xcolor}
\usepackage{hyperref}
\hypersetup{
  colorlinks=true,
  linkcolor=blue,
  citecolor=blue,
  urlcolor=cyan
}

\newcommand{\R}{\mathbb R}
\newcommand{\Z}{\mathbb Z}
\newcommand{\Sph}{\mathbb S}
\newcommand{\Ccal}{\mathcal C}
\newcommand{\Hcal}{\mathcal H}
\newcommand{\Pp}{\mathbf P}
\newcommand{\Ee}{\mathbf E}
\newcommand{\Vol}{\operatorname{Vol}}

\newtheorem{theorem}{Theorem}[section]
\newtheorem{proposition}[theorem]{Proposition}
\newtheorem{lemma}[theorem]{Lemma}
\newtheorem{corollary}[theorem]{Corollary}
\theoremstyle{definition}
\newtheorem{definition}[theorem]{Definition}
\newtheorem{example}[theorem]{Example}
\newtheorem{remark}[theorem]{Remark}

\begin{document}

\title[Brownian Loops and Asymptotic Cycles]
{Brownian Loops, Singular Homology, and Asymptotic Cycles}

\author[Alberto Verjovsky]{Alberto Verjovsky}
\address{Instituto de Matem\'aticas\\
Universidad Nacional Aut\'onoma de M\'exico\\
Apartado Postal 273, Administraci\'on de Correos \#3\\
C.P. 62251 Cuernavaca, Morelos, M\'exico}
\email{alberto@matcuer.unam.mx}

\author[Ricardo F. Vila-Freyer]{Ricardo F. Vila-Freyer}
\address{CIMAT\\
Apdo. Postal 402, Guanajuato, Gto., C.P. 36000, M\'exico}
\email{vila@cimat.mx}

\date{\today}

\keywords{Brownian motion, asymptotic cycles, singular homology, stochastic currents, large deviations, Hodge theory}
\subjclass[2020]{60J65, 58J65, 57R19, 60F05}

\begin{abstract}
Let $M$ be a closed connected Riemannian manifold.  A continuous
semimartingale segment can be closed by a Borel family of paths of uniformly
bounded length, producing a singular homology class
$H_t\in H_1(M;\R)$.  For every linear choice of smooth closed representatives
of $H^1(M;\R)$, the corresponding Stratonovich homology differs from $H_t$ by
a uniformly bounded term, almost surely and uniformly in time.  The classes
$H_t$ are additive under time shift up to a uniformly bounded error.

For Brownian motion this comparison yields the Gaussian central limit theorem
with covariance given by the normalized Hodge inner product, a functional
central limit theorem for the polygonal interpolation of the closed homology
classes, and the almost sure limit $H_t/t\to0$.  For elliptic diffusions, the
large-deviation principle of Galkin--Mariani passes to $H_T/T$ with the same
rate function.  The construction uses Schwartzman's closing procedure and is independent, at these asymptotic scales, of the chosen bounded closing family.
\end{abstract}

\maketitle
\tableofcontents

\section{Introduction}

Schwartzman's theory of asymptotic cycles associates homology classes to long pieces of trajectories of a measure-preserving flow \cite{Schwartzman}.  If $f_t:M\to M$ is a flow and one closes the orbit segment
\[
   \{f_s(x):0\le s\le T\}
\]
by a path of bounded length, the resulting cycle can often be divided by $T$ and made to converge in $H_1(M;\R)$.  The limit records the average motion of the orbit in homology.

The same closing construction applies to a continuous semimartingale.  Here a continuous semimartingale on a manifold is understood in the standard chartwise sense: in local coordinates it is the sum of a continuous local martingale and a continuous finite-variation process; see \cite{Emery}.  Closing a segment by a path of uniformly bounded length gives an ordinary singular homology class.  Its pairing with a cohomology class differs by a uniformly bounded term from the Stratonovich integral of any closed representative of that class.  Brownian motion is the main example: on a closed Riemannian manifold its homological fluctuations occur at scale $\sqrt t$, whereas the linear Schwartzman normalization tends to zero.

Stochastic line integrals and the homological behavior of diffusions have been studied by Ikeda--Manabe \cite{IkedaManabe}, Manabe \cite{Manabe}, Ochi \cite{Ochi}, and Kuwada \cite{Kuwada}.  Watanabe considered Brownian winding on Riemann surfaces \cite{Watanabe}.  Geng and Iyer proved Gaussian limit laws for abelianized winding of reflected Brownian motion on compact manifolds with boundary \cite{GengIyer}, and Galkin and Mariani studied large deviations and rigidity for random homology of diffusion processes \cite{GalkinMariani}.  The construction used here keeps Schwartzman's closing procedure and compares the resulting singular homology classes with the homology defined by stochastic integration.

Brownian motion also arises from random geodesic motion.  J{\o}rgensen proved an invariance principle for geodesic random walks \cite{Jorgensen}, and Pinsky constructed an isotropic transport process on the tangent bundle whose transition semigroup converges, after rescaling, to the Brownian semigroup \cite{Pinsky}.  Before the limiting process is taken, the paths are made from ordinary geodesic arcs and closed $1$-forms are integrated along piecewise smooth curves, as in the construction of asymptotic cycles.

A related classical problem is the winding of planar Brownian motion.  Spitzer proved that the winding angle of planar Brownian motion around the origin, divided by $\log t$, converges to a Cauchy distribution \cite{Spitzer}; see also Durrett \cite{Durrett} and Pitman--Yor \cite{PitmanYor}.  The punctured plane leads to a different time change.  If
\[
   Z_t=R_t e^{i\Theta_t}
\]
is planar Brownian motion, then the angular process can be represented through a random clock involving
\[
   \int_0^t \frac{ds}{R_s^2},
\]
and $R_t$ is a two-dimensional Bessel process.  The corresponding Bessel clock is singular and leads to the logarithmic scaling in Spitzer's theorem.

The situation considered here is different.  Let $\omega$ be a harmonic $1$-form on a closed Riemannian manifold and $X_t$ Brownian motion.  The Stratonovich integral
\[
   J_t^\omega=\int_0^t \omega\circ dX_s
\]
is a continuous martingale and
\[
   \langle J^\omega\rangle_t
   =\int_0^t |\omega|^2(X_s)\,ds.
\]
Under the stationary Wiener measure the ergodic theorem gives
\[
   \frac{1}{t}\langle J^\omega\rangle_t
   \longrightarrow
   \frac{1}{\Vol(M)}\int_M |\omega|^2\,dv_g.
\]
Thus the clock is asymptotically linear.  The martingale central limit theorem then gives a Gaussian limit at scale $\sqrt t$.

For $\alpha,\beta\in H^1(M;\R)$ let $\omega_\alpha,\omega_\beta$ be their harmonic representatives and define
\begin{equation}\label{eq:Qintro}
   Q(\alpha,\beta)
   =\frac{1}{\Vol(M)}
   \int_M\langle\omega_\alpha,\omega_\beta\rangle\,dv_g.
\end{equation}
For Brownian motion, the rescaled singular homology converges to the centered Gaussian measure on $H_1(M;\R)$ whose covariance, under the canonical pairing with $H^1(M;\R)$, is $Q$.

At these scales the limits are independent of the bounded closing family, and the Schwartzman normalization vanishes:
\[
   \frac{H_t}{t}\longrightarrow0
   \qquad\text{almost surely}.
\]

Brownian motion is taken with infinitesimal generator $\frac12\Delta_g$; this fixes the constants in the covariance formula.  If Brownian motion is defined with generator $\Delta_g$, the covariance below is multiplied by $2$.

\section{Brownian motion and Wiener measure}

Throughout, $M$ is a closed connected smooth Riemannian manifold with metric $g$.  Compactness implies stochastic completeness.

Let
\[
   \Ccal(M)=C([0,\infty),M)
\]
with its Borel $\sigma$-algebra.  For $x\in M$, let $\Pp_x$ be Wiener measure for Brownian motion $X_t$ starting at $x$ and having generator $\frac12\Delta_g$.  Let
\[
   \mu=\frac{dv_g}{\Vol(M)}
\]
be normalized Riemannian volume and define
\begin{equation}\label{eq:stationaryWiener}
   \Pp_\mu(A)=\int_M \Pp_x(A)\,d\mu(x),
   \qquad A\subset\Ccal(M)\ \text{Borel}.
\end{equation}
Thus $X_0$ has distribution $\mu$.

For $s\ge0$ let $\theta_s:\Ccal(M)\to\Ccal(M)$ be the time shift
\[
   (\theta_sX)(t)=X_{s+t}.
\]
Since Riemannian volume is invariant for the heat semigroup, $\Pp_\mu$ is invariant under $\theta_s$.  Since $M$ is connected and the heat kernel is strictly positive, the stationary Brownian motion is ergodic; see \cite{Hsu} for the heat-kernel facts used here.  Consequently, for every $f\in L^1(M,\mu)$,
\begin{equation}\label{eq:ergodic}
   \frac1t\int_0^t f(X_s)\,ds
   \longrightarrow \int_M f\,d\mu
   \qquad \Pp_\mu\text{-almost surely}.
\end{equation}
Only bounded smooth functions will be used below.  The same ergodic averages hold when Brownian motion starts at a prescribed point.

\begin{proposition}\label{prop:ergodic-start}
For every $x\in M$ and every bounded Borel function $f:M\to\R$,
\begin{equation}\label{eq:ergodicx}
   \frac1t\int_0^t f(X_s)\,ds
   \longrightarrow \int_M f\,d\mu
   \qquad \Pp_x\text{-almost surely}.
\end{equation}
\end{proposition}

\begin{proof}
Let $A_f$ be the event in \eqref{eq:ergodic}.  Since $\Pp_\mu(A_f)=1$,
\[
   \Pp_y(A_f)=1
\]
for $\mu$-almost every $y\in M$.  Fix $x\in M$ and $\varepsilon>0$.  The law of $X_\varepsilon$ under $\Pp_x$ has the heat-kernel density $p_\varepsilon(x,y)$ with respect to Riemannian volume and is therefore absolutely continuous with respect to $\mu$; for these standard facts about Brownian motion on a Riemannian manifold, see \cite{Hsu}.  By the Markov property,
\[
   \Pp_x(\theta_\varepsilon^{-1}A_f)=1.
\]
On this event,
\[
 \frac1t\int_\varepsilon^t f(X_s)\,ds
 \longrightarrow \int_M f\,d\mu.
\]
Since $f$ is bounded, the contribution of $[0,\varepsilon]$ divided by $t$ tends to zero.  This proves \eqref{eq:ergodicx}.
\end{proof}

For a smooth $1$-form $\omega$, write
\[
   \int_0^t\omega\circ dX_s.
\]
The Stratonovich integral satisfies the ordinary chain rule:
\begin{equation}\label{eq:chainrule}
   \int_0^t df\circ dX_s=f(X_t)-f(X_0)
\end{equation}
for $f\in C^\infty(M)$.  The intrinsic construction may be found in It\^o \cite{Ito1963}, Hsu \cite{Hsu}, and \`Emery \cite{Emery}.

With the sign convention $\delta df=-\Delta_g f$, the Stratonovich integral has the semimartingale decomposition
\begin{equation}\label{eq:stratito}
   \int_0^t\omega\circ dX_s
   =N_t^\omega-\frac12\int_0^t(\delta\omega)(X_s)\,ds,
\end{equation}
where $N_t^\omega$ is a continuous local martingale.  Its covariation with $N_t^\eta$ is
\begin{equation}\label{eq:crossvariation}
   \langle N^\omega,N^\eta\rangle_t
   =\int_0^t\langle\omega,\eta\rangle(X_s)\,ds.
\end{equation}
In particular, if $\omega$ is harmonic, then $d\omega=0$ and $\delta\omega=0$, so that
\begin{equation}\label{eq:harmmart}
   J_t^\omega:=\int_0^t\omega\circ dX_s
\end{equation}
is a continuous local martingale.  Since $M$ is compact, $\omega$ is bounded, and hence
\[
\mathbf E_x\langle J^\omega\rangle_t
\le t\,\|\omega\|_\infty^2<\infty
\qquad (x\in M).
\]
Therefore $J_t^\omega$ is a square-integrable martingale on every finite time interval, and
\begin{equation}\label{eq:qv}
   \langle J^\omega,J^\eta\rangle_t
   =\int_0^t\langle\omega,\eta\rangle(X_s)\,ds
\end{equation}
for harmonic $\omega,\eta$.

\section{Geodesic approximations of Brownian motion}\label{sec:geodesic-approx}

J{\o}rgensen considered random walks whose successive steps are geodesic segments and proved convergence, under diffusive rescaling and suitable assumptions on the distribution of the steps, to diffusion processes on the manifold; isotropic identically distributed steps give Brownian motion \cite{Jorgensen}.  Pinsky constructed an isotropic transport process on the tangent bundle of a complete Riemannian manifold.  On a compact manifold, after introducing a small parameter, its transition semigroup converges to the Brownian semigroup \cite{Pinsky}.  Between successive changes of direction the motion is geodesic.

Let $c:[0,t]\to M$ be a piecewise smooth path and close it by a path $\gamma_{c(t),c(0)}$.  Put
\[
   \Gamma(c,t)=c*\gamma_{c(t),c(0)}.
\]
For every smooth closed $1$-form $\omega$,
\begin{equation}\label{eq:broken-geodesic-pairing}
   \bigl\langle[\omega],[\Gamma(c,t)]\bigr\rangle
   =\int_c\omega+\int_{\gamma_{c(t),c(0)}}\omega.
\end{equation}
For a geodesic random walk, therefore, the homology of the closed path is computed by ordinary line integrals along its geodesic pieces and along the closing path.

The Stratonovich integral has the ordinary chain rule and agrees with the usual line integral on paths of finite variation.  It is also the integral obtained from the standard smooth approximations of stochastic paths; see \cite{Emery,Hsu}.  Formula \eqref{eq:broken-geodesic-pairing} is consequently the finite-variation form of the stochastic pairing used below.

Pinsky's construction takes place on the tangent bundle.  Between two changes of direction the tangent vector evolves by the geodesic flow; the random choice of new directions produces the diffusive limit on the base manifold.  Thus the closed Brownian paths considered below can be approached by closed broken geodesics, to which the usual construction of asymptotic cycles applies directly.

\section{Closing Brownian paths and the homology process}

The closing paths will be chosen with uniformly bounded length and Borel dependence on the endpoints.

\begin{lemma}[A measurable family of bounded closing paths]\label{lem:closing}
There is a Borel family
\[
   (x,y)\longmapsto\gamma_{x,y}:[0,1]\to M
\]
of piecewise smooth paths joining $x$ to $y$ and a constant $L<\infty$ such that
\[
   \operatorname{length}(\gamma_{x,y})\le L
\]
for all $x,y\in M$.
\end{lemma}

\begin{proof}
Choose a finite cover $U_1,\ldots,U_m$ by geodesically convex open sets and points $p_i\in U_i$.  Define a Borel partition subordinate to this cover by
\[
V_1=U_1,\qquad
V_i=U_i\setminus\bigcup_{j<i}U_j,\quad 2\le i\le m.
\]
Then $V_i\subset U_i$ for every $i$.  For each pair $(i,j)$ choose once and for all a smooth path $c_{ij}$ from $p_i$ to $p_j$.  If $x\in V_i$ and $y\in V_j$, let $\gamma_{x,y}$ be the concatenation of the unique minimizing geodesic in $U_i$ from $x$ to $p_i$, the path $c_{ij}$, and the unique minimizing geodesic in $U_j$ from $p_j$ to $y$.  Inside each geodesically convex set the minimizing geodesic depends smoothly on its endpoints.  Hence the dependence on $(x,y)$ is Borel after restriction to the Borel pieces $V_i\times V_j$.  The lengths are uniformly bounded because there are only finitely many sets and finitely many paths $c_{ij}$.
\end{proof}

Fix such a family.  For any continuous path $X\in\Ccal(M)$ and $t>0$, let $\Gamma(X,t)$ be the closed continuous path obtained by following $X_s$, $0\le s\le t$, and then following $\gamma_{X_t,X_0}$ from $X_t$ back to $X_0$.  After a linear reparametrization of the two pieces, $\Gamma(X,t)$ is a singular $1$-cycle.  We define
\begin{definition}\label{def:Ht}
The \emph{Brownian homology process} is
\[
   H_t(X):=[\Gamma(X,t)]\in H_1(M;\R),
   \qquad t\ge0,
\]
with $H_0=0$.
\end{definition}

The class of a loop is integral; hence $H_t$ takes values in the image of
$H_1(M;\Z)/\operatorname{tors}$ inside $H_1(M;\R)$, a lattice when
$b_1(M)>0$.

The loop $\Gamma(X,t)$ defines a singular homology class because singular simplices are only required to be continuous; rectifiability is not required.  Its pairing with smooth de Rham classes will be expressed by stochastic integration.

Let $\langle\cdot,\cdot\rangle$ denote the canonical pairing
\[
   H^1(M;\R)\times H_1(M;\R)\longrightarrow\R.
\]
Equivalently, since the coefficients are in a field,
\[
   H_1(M;\R)\simeq H^1(M;\R)^*.
\]
This is the universal-coefficient pairing.  If $M$ is oriented, Poincar\'e duality gives instead an isomorphism $H^1(M;\R)\simeq H_{n-1}(M;\R)$.

For a smooth closed $1$-form $\omega$, define
\begin{equation}\label{eq:stochpair}
   I_t(\omega)
   =\int_0^t\omega\circ dX_s
    +\int_{\gamma_{X_t,X_0}}\omega.
\end{equation}
The second integral is an ordinary line integral along the piecewise smooth closing path.

\begin{lemma}[Circle-valued semimartingales and lifts]\label{lem:circlelift}
Let $\phi:M\to\Sph^1$ be smooth and let $X_s$, $0\le s\le t$, be a continuous semimartingale.  Put $Y_s=\phi(X_s)$ and let $\widetilde Y_s$ be any continuous lift of $Y_s$ to $\R$ under $u\mapsto e^{iu}$.  Then
\[
 \int_0^t \phi^*\!\left(\frac{d\theta}{2\pi}\right)\circ dX_s
 =\frac{\widetilde Y_t-\widetilde Y_0}{2\pi}.
\]
\end{lemma}

\begin{proof}
The assertion is local on the circle.  On every proper open arc $J\subset\Sph^1$ choose a smooth branch $\vartheta:J\to\R$ of the angular coordinate.  If $Y_s\in J$ for $a\le s\le b$, the Stratonovich chain rule gives
\[
 \int_a^b \phi^*\!\left(\frac{d\theta}{2\pi}\right)\circ dX_s
 =\frac{\vartheta(Y_b)-\vartheta(Y_a)}{2\pi}.
\]
Cover $\Sph^1$ by finitely many such arcs.  By uniform continuity of $Y$ on $[0,t]$, subdivide the interval so that the image of each subinterval lies in one arc.  Summing the identities makes the intermediate angular increments telescope and gives the increment of the continuous lift.
\end{proof}

\begin{proposition}[Stochastic representation of the homology pairing]\label{prop:pairing}
Let $X$ be a continuous semimartingale on $M$.  For every smooth closed
$1$-form $\omega$,
\begin{equation}\label{eq:pairingidentity}
   I_t(\omega)=\langle[\omega],H_t\rangle
\end{equation}
almost surely, simultaneously for all $t\ge0$, after choosing the usual
continuous version of the Stratonovich integral.  In particular,
$I_t(\omega)$ depends only on the de Rham class of $\omega$.
\end{proposition}

\begin{proof}
First suppose that $\omega=df$.  By the Stratonovich chain rule,
\[
   \int_0^t df\circ dX_s=f(X_t)-f(X_0),
\]
while
\[
   \int_{\gamma_{X_t,X_0}}df=f(X_0)-f(X_t).
\]
Thus $I_t(df)=0$, so $I_t$ depends only on the de Rham class.

The image of $H^1(M;\Z)$ in $H^1(M;\R)$ is a full lattice and therefore spans $H^1(M;\R)$.  Let $\alpha\in H^1(M;\Z)$ and represent it by a smooth map $\phi:M\to\Sph^1$.  The form $\phi^*(d\theta/(2\pi))$ represents the real image of $\alpha$.  Lemma~\ref{lem:circlelift} identifies the stochastic integral over the Brownian part with the lifted angular increment of $\phi(X_s)$.  The ordinary integral over the closing path gives the remaining angular increment.  Their sum is therefore the total angular change of the closed loop $\phi\circ\Gamma(X,t)$, divided by $2\pi$, namely its degree.  Hence
\[
 I_t\!\left(\phi^*\frac{d\theta}{2\pi}\right)
 =\deg(\phi\circ\Gamma(X,t))
 =\langle\alpha,[\Gamma(X,t)]\rangle.
\]
The identity follows for integral classes and hence, by linearity, for all classes in $H^1(M;\R)$.
\end{proof}

\begin{corollary}[Measurability of the homology process]\label{cor:measurableHt}
For each fixed $t\ge0$, the map $X\mapsto H_t(X)\in H_1(M;\R)$ is measurable on Wiener path space.
\end{corollary}

\begin{proof}
Choose a basis of $H^1(M;\R)$ consisting of real images of integral classes,
say $\alpha_1,\ldots,\alpha_{b_1}$.  By Proposition~\ref{prop:pairing}, each coordinate
\[
\langle\alpha_j,H_t\rangle
\]
is the sum of a Stratonovich stochastic integral and the ordinary integral over the Borel family of closing paths.  On each Borel piece $V_i\times V_j$, the two geodesic parts of the closing path depend smoothly on the endpoints and the middle path $c_{ij}$ is fixed; hence
$(x,y)\mapsto\int_{\gamma_{x,y}}\omega$ is Borel for every smooth $1$-form $\omega$.
Thus these coordinates are measurable and determine $H_t$ in the finite-dimensional space $H_1(M;\R)$.
\end{proof}

\section{Closed singular cycles and stochastic homology}\label{sec:realization}

The construction above applies to an arbitrary continuous semimartingale.

Let $Z^1(M)$ denote the vector space of smooth closed $1$-forms, and let
\[
   \xi:H^1(M;\R)\longrightarrow Z^1(M)
\]
be a linear map such that $\xi_\alpha$ is a smooth closed representative of
$\alpha$ for every $\alpha\in H^1(M;\R)$.  For a continuous semimartingale
$X$ define $\mathcal J_t^\xi\in H_1(M;\R)$ by
\begin{equation}\label{eq:Jxi}
   \langle\alpha,\mathcal J_t^\xi\rangle
   =\int_0^t \xi_\alpha\circ dX_s .
\end{equation}
The map $\alpha\mapsto\xi_\alpha$ is linear, so this defines an element
of $H^1(M;\R)^*\simeq H_1(M;\R)$.

\begin{theorem}[Uniform comparison]\label{thm:bounded-realization}
Let $X$ be any continuous semimartingale on the closed manifold $M$, and let
$H_t$ be obtained by closing $X_{[0,t]}$ with the family of Lemma~\ref{lem:closing}.
For every linear choice $\xi$ of smooth closed representatives there is a
constant $C_\xi<\infty$, depending only on $\xi$ and on the closing family,
such that, outside one null set,
\begin{equation}\label{eq:bounded-realization}
   \|H_t-\mathcal J_t^\xi\|\le C_\xi
   \qquad\text{for every }t\ge0.
\end{equation}
\end{theorem}

\begin{proof}
For every $\alpha\in H^1(M;\R)$, Proposition~\ref{prop:pairing} applied to the
closed form $\xi_\alpha$ gives
\begin{equation}\label{eq:bounded-realization-pair}
   \langle\alpha,H_t-\mathcal J_t^\xi\rangle
   =\int_{\gamma_{X_t,X_0}}\xi_\alpha .
\end{equation}
Choose a norm $\|\cdot\|_*$ on $H^1(M;\R)$.  Since the space is finite
dimensional and $\xi$ is linear, there is $K_\xi<\infty$ such that
\[
   \|\xi_\alpha\|_\infty\le K_\xi\|\alpha\|_*
\]
for all $\alpha$.  Hence
\[
   |\langle\alpha,H_t-\mathcal J_t^\xi\rangle|
   \le L K_\xi\|\alpha\|_* .
\]
Taking the dual norm on $H_1(M;\R)$ proves \eqref{eq:bounded-realization}.
\end{proof}

The stochastic homology is additive under time shift.  The closed singular classes satisfy the corresponding identity up to a uniformly bounded error.

\begin{corollary}[Additivity up to bounded error]\label{cor:quasicocycle}
For every continuous semimartingale path and all $s,t\ge0$,
\begin{equation}\label{eq:quasicocycle}
 \bigl\|H_{s+t}(X)-H_s(X)-H_t(\theta_sX)\bigr\|\le 3C_\xi .
\end{equation}
The bound is independent of $s$ and $t$.
\end{corollary}

\begin{proof}
Stratonovich integration is additive under concatenation, hence
\[
   \mathcal J_{s+t}^\xi(X)
   =\mathcal J_s^\xi(X)+\mathcal J_t^\xi(\theta_sX).
\]
Subtract this identity from the left-hand side of \eqref{eq:quasicocycle} and
apply Theorem~\ref{thm:bounded-realization} to the three terms.
\end{proof}

\begin{corollary}[Independence of the closing family]\label{cor:closing-universality}
If $H_t$ and $\widetilde H_t$ are obtained from two Borel closing families of
uniformly bounded length, then
\[
   \sup_{t\ge0}\|H_t-\widetilde H_t\|<\infty .
\]
Every asymptotic statement unchanged by a bounded perturbation is therefore independent of the chosen closing family.
\end{corollary}

\begin{proof}
Apply Theorem~\ref{thm:bounded-realization} to both closing rules and the same
section $\xi$.
\end{proof}

For completeness, a family of random variables $Y_t$ in a finite-dimensional vector space is said to satisfy a large deviation principle with speed $t$ and rate function $I$ if, for every open set $O$ and every closed set $F$,
\[
 -\inf_{y\in O} I(y)
 \le \liminf_{t\to\infty}\frac1t\log\mathbf P(Y_t\in O),
 \qquad
 \limsup_{t\to\infty}\frac1t\log\mathbf P(Y_t\in F)
 \le -\inf_{y\in F} I(y).
\]
The rate function is called \emph{good} when each sublevel set $\{I\le a\}$ is compact.  We use the terminology and results of Dembo--Zeitouni \cite{DemboZeitouni}.

\begin{proposition}[Consequences under normalization]\label{prop:transfer}
Let $a_t>0$ with $a_t\to\infty$.  Then
\begin{equation}\label{eq:transfer}
   \left\|\frac{H_t}{a_t}-\frac{\mathcal J_t^\xi}{a_t}\right\|
   \le \frac{C_\xi}{a_t}\longrightarrow0
\end{equation}
almost surely.  Hence the two normalized families have the same limits in
probability, in distribution, and almost surely whenever one of these limits
exists.  Moreover, if $\mathcal J_t^\xi/t$ satisfies a large deviation
principle with speed $t$ and good rate function $I$, then $H_t/t$ satisfies
the same large deviation principle with the same rate function.
\end{proposition}

\begin{proof}
For every $\varepsilon>0$,
\[
 \mathbf P\!\left(
 \left\|\frac{H_t}{t}-\frac{\mathcal J_t^\xi}{t}\right\|>\varepsilon
 \right)=0
\]
for all $t>C_\xi/\varepsilon$.  Thus the two families are exponentially equivalent.  
The standard exponential-equivalence theorem for large 
deviations gives the claim \cite[Theorem~4.2.13]{DemboZeitouni}.
\end{proof}

\begin{corollary}[Large deviations for closed singular cycles]\label{cor:GM-LDP}
Let $X_t$ be the diffusion on $M$ with generator
\[
   Lf=\frac12\Delta_g f+\langle b,df\rangle,
\]
where $b$ is a smooth vector field.  Let $G$ be the good rate function for
random homology defined by Galkin--Mariani \cite{GalkinMariani}.  Then
\[
   \frac{H_T}{T}
\]
satisfies a good large deviation principle with speed $T$ and the same rate function
$G$.  This conclusion is independent of the uniformly bounded closing rule.
\end{corollary}

\begin{proof}
Galkin--Mariani define, after choosing a linear section
$c\mapsto\xi_c$ of closed representatives,
\[
   \langle h_T,c\rangle
   =\frac1T\int_0^T\xi_c(X_s)\circ dX_s
   =\left\langle\frac{\mathcal J_T^\xi}{T},c\right\rangle
\]
and prove (Remark~2.2 in \cite{GalkinMariani}) that $h_T$ satisfies a good large deviation principle with speed
$T$ and rate $G$, independently of the section.  Proposition~\ref{prop:transfer}
therefore transfers that principle to $H_T/T$.
\end{proof}

\begin{remark}
The rate function $G$ governs the large deviations of the singular homology classes obtained by closing long diffusion paths.  The same closing procedure appears in Schwartzman's theory of asymptotic cycles.
\end{remark}

\subsection{Harmonic circle maps and homology coordinates}

Let $\alpha\in H^1(M;\Z)$ and let $\omega_\alpha$ be the harmonic representative of its real cohomology class.  Since $\omega_\alpha$ has integral periods, there is a smooth map
\[
   \phi_\alpha:M\longrightarrow \Sph^1
\]
such that
\[
   \omega_\alpha=\phi_\alpha^*\!\left(\frac{d\theta}{2\pi}\right).
\]
This is the standard circle-valued realization of an integral cohomology class; see \cite{BairdWood}.  The map $\phi_\alpha$ is harmonic.  Since the target is one-dimensional, a harmonic map to $\Sph^1$ is horizontally conformal at every point at which its differential is nonzero; hence it is a harmonic morphism in the usual sense; see \cite{BairdWood}.

If $\widetilde{\phi_\alpha(X_t)}$ is a continuous lift to $\R$, Lemma~\ref{lem:circlelift} gives
\begin{equation}\label{eq:harmonic-circle-coordinate}
 \int_0^t\omega_\alpha\circ dX_s
 =\frac{\widetilde{\phi_\alpha(X_t)}-\widetilde{\phi_\alpha(X_0)}}{2\pi}.
\end{equation}
Thus the $\alpha$-coordinate of the stochastic homology is the lifted angular displacement of the Brownian path under the harmonic circle map representing $\alpha$.  For Brownian motion the quadratic variation of the right-hand side is
\[
 A_t^\alpha=\int_0^t |\omega_\alpha|^2(X_s)\,ds.
\]
For every $x\in M$, Proposition~\ref{prop:ergodic-start} gives
\[
 \frac{A_t^\alpha}{t}\longrightarrow Q(\alpha,\alpha)
 \qquad \Pp_x\text{-almost surely}.
\]
The corresponding time change is therefore asymptotically linear.  This is the circle-valued form of the martingale argument used below.

\section{The Gaussian limit}

Let $b_1=\dim H^1(M;\R)$.  By the Hodge theorem, each class $\alpha\in H^1(M;\R)$ has a unique harmonic representative, denoted by $\omega_\alpha$; see, for example, \cite[Chapter~3]{Jost}.  Define
\begin{equation}\label{eq:Q}
   Q(\alpha,\beta)
   =\int_M\langle\omega_\alpha,\omega_\beta\rangle\,d\mu
   =\frac{1}{\Vol(M)}
     \int_M\langle\omega_\alpha,\omega_\beta\rangle\,dv_g.
\end{equation}
This is a positive-definite inner product on $H^1(M;\R)$.

\begin{remark}
If $b_1(M)=0$, then $H^1(M;\R)=H_1(M;\R)=0$, so the limit statements below reduce to the zero class.
\end{remark}

The continuous-martingale central limit theorem will be used in the following form: if $M_t$ is a continuous local martingale and
\[
   \frac{\langle M\rangle_t}{t}\longrightarrow\sigma^2
\]
in probability, then
\[
   \frac{M_t}{\sqrt t}\Longrightarrow N(0,\sigma^2).
\]
The multidimensional functional form follows from convergence of the matrix of quadratic covariations.  See Rebolledo \cite{Rebolledo} and Revuz--Yor \cite{RevuzYor}.

\begin{theorem}[Central limit theorem for Brownian homology]\label{thm:main}
Let $M$ be a closed connected Riemannian manifold.  For every $x\in M$, the Brownian homology process under $\Pp_x$ satisfies
\begin{equation}\label{eq:mainCLT}
   \frac{H_t}{\sqrt t}\Longrightarrow G
   \qquad\text{in }H_1(M;\R),
\end{equation}
where $G$ is the centered Gaussian random vector characterized by
\begin{equation}\label{eq:covG}
   \operatorname{Cov}
   \bigl(\langle\alpha,G\rangle,
         \langle\beta,G\rangle\bigr)
   =Q(\alpha,\beta),
   \qquad \alpha,\beta\in H^1(M;\R).
\end{equation}
Equivalently, for every $\alpha\in H^1(M;\R)$,
\begin{equation}\label{eq:scalarCLT}
   \frac{\langle\alpha,H_t\rangle}{\sqrt t}
   \Longrightarrow
   N\bigl(0,Q(\alpha,\alpha)\bigr).
\end{equation}
If $\alpha\ne0$, then $Q(\alpha,\alpha)>0$.  The same conclusion holds for Brownian motion with any initial distribution on $M$.
\end{theorem}

\begin{proof}
Fix $x\in M$ and $\alpha\in H^1(M;\R)$, and let $\omega=\omega_\alpha$ be its harmonic representative.  Proposition \ref{prop:pairing} gives
\begin{equation}\label{eq:decomp}
   \langle\alpha,H_t\rangle
   =J_t^\omega+R_t^\omega,
\end{equation}
where
\[
   J_t^\omega=\int_0^t\omega\circ dX_s
\]
and
\[
   R_t^\omega=\int_{\gamma_{X_t,X_0}}\omega.
\]
Since $\delta\omega=0$, $J_t^\omega$ is a continuous martingale.  Since the closing paths have length at most $L$,
\begin{equation}\label{eq:closingbound}
   |R_t^\omega|
   \le L\,\|\omega\|_\infty
\end{equation}
for every $t$.

The quadratic variation is
\[
   \langle J^\omega\rangle_t
   =\int_0^t|\omega|^2(X_s)\,ds.
\]
By Proposition~\ref{prop:ergodic-start},
\begin{equation}\label{eq:qvlimit}
   \frac{1}{t}\langle J^\omega\rangle_t
   \longrightarrow
   \int_M|\omega|^2\,d\mu
   =Q(\alpha,\alpha)
\end{equation}
almost surely.  The martingale central limit theorem gives
\[
   \frac{J_t^\omega}{\sqrt t}
   \Longrightarrow N\bigl(0,Q(\alpha,\alpha)\bigr).
\]
The bound \eqref{eq:closingbound} implies $R_t^\omega/\sqrt t\to0$ uniformly, and \eqref{eq:scalarCLT} follows.

For the joint law, let $\alpha_1,\ldots,\alpha_k\in H^1(M;\R)$ and let $a_1,\ldots,a_k\in\R$.  By linearity,
\[
  \sum_{j=1}^k a_j\langle\alpha_j,H_t\rangle
  =\left\langle\sum_{j=1}^k a_j\alpha_j,H_t\right\rangle.
\]
The scalar result applied to $\sum_j a_j\alpha_j$ gives the required Gaussian limit for every linear combination.  By the Cram\'er--Wold theorem, convergence of all real linear combinations implies convergence of the vector-valued laws; see \cite{Kallenberg}.  This gives the multivariate convergence and the covariance formula \eqref{eq:covG}.

Finally, if $\alpha\ne0$, its harmonic representative is not identically zero, and therefore
\[
   Q(\alpha,\alpha)=\int_M|\omega_\alpha|^2\,d\mu>0.
\]
\end{proof}

For harmonic representatives, define the $H_1(M;\R)$-valued martingale

\begin{definition}[The harmonic homology martingale]\label{def:harmonicmart}
For every $t\ge0$ define $\mathcal M_t\in H_1(M;\R)$ by
\begin{equation}\label{eq:harmonicmartdef}
   \langle\alpha,\mathcal M_t\rangle
   =\int_0^t\omega_\alpha\circ dX_s,
   \qquad \alpha\in H^1(M;\R),
\end{equation}
where $\omega_\alpha$ is the harmonic representative of $\alpha$.
\end{definition}

The map $\alpha\mapsto\int_0^t\omega_\alpha\circ dX_s$ is linear, so
\eqref{eq:harmonicmartdef} defines an element of
$H^1(M;\R)^*\simeq H_1(M;\R)$.  The process $\mathcal M_t$ has continuous
paths.  It is the special case of $\mathcal J_t^\xi$ in
Theorem~\ref{thm:bounded-realization} obtained by taking
$\xi_\alpha=\omega_\alpha$.  Hence, for a deterministic constant $C$,
\begin{equation}\label{eq:norm-bounded-distance}
   \|H_t-\mathcal M_t\|\le C
   \qquad\text{for all }t\ge0
\end{equation}
almost surely, simultaneously for all $t$.

\begin{theorem}[Functional central limit theorem]\label{thm:functional}
For $T>0$ define, on $0\le s\le1$,
\[
   \Hcal_T(s)=\frac{\mathcal M_{Ts}}{\sqrt T}.
\]
For every $x\in M$, viewed under $\Pp_x$ as a random element of $C([0,1],H_1(M;\R))$, the process $\Hcal_T$ converges in distribution, as $T\to\infty$, to Brownian motion in $H_1(M;\R)$ with covariance form $Q$.  In particular, for $0\le r\le s\le1$ and $\alpha,\beta\in H^1(M;\R)$,
\[
   \Ee\bigl[
   \langle\alpha,G_s-G_r\rangle
   \langle\beta,G_s-G_r\rangle
   \bigr]
   =(s-r)Q(\alpha,\beta).
\]
The same convergence holds for Brownian motion with any initial distribution on $M$.
\end{theorem}

\begin{proof}
Fix $x\in M$.  Choose harmonic forms $\omega_1,\ldots,\omega_{b_1}$ forming a basis of $H^1(M;\R)$.  In the dual coordinates the process $\mathcal M_t$ is the vector martingale
\[
   M_t=(J_t^{\omega_1},\ldots,J_t^{\omega_{b_1}}).
\]
Its matrix quadratic variation is
\[
   \langle J^i,J^j\rangle_t
   =\int_0^t
   \langle\omega_i,\omega_j\rangle(X_u)\,du.
\]
For $0\le s\le1$,
\[
  \frac1T\langle J^i,J^j\rangle_{Ts}
  =\frac1T\int_0^{Ts}
  \langle\omega_i,\omega_j\rangle(X_u)\,du.
\]
By Proposition~\ref{prop:ergodic-start}, this converges $\Pp_x$-almost surely to
\[
   sQ(\alpha_i,\alpha_j).
\]
The convergence is uniform in $s\in[0,1]$: the functions on the left are uniformly Lipschitz because the integrands are bounded, and convergence holds almost surely on all rational $s$ after taking a countable intersection of full-measure sets.  Uniform equicontinuity then gives uniform convergence on $[0,1]$.  For the rescaled martingale $J_s^{(T)}=T^{-1/2}J_{Ts}$,
\[
 \langle J^{(T),i},J^{(T),j}\rangle_s
 =\frac1T\langle J^i,J^j\rangle_{Ts}
 \longrightarrow sQ(\alpha_i,\alpha_j)
\]
uniformly in $s$, almost surely and hence in probability.  The functional martingale central limit theorem then gives the asserted convergence; see Rebolledo \cite{Rebolledo}.
\end{proof}

\begin{remark}
The limiting vector $G$ in Theorem~\ref{thm:main} has the same law as the time-one value $G_1$ of the limiting Brownian motion in Theorem~\ref{thm:functional}.  Both are centered Gaussian random vectors on $H_1(M;\R)$ with covariance form $Q$.
\end{remark}

\begin{theorem}[Functional CLT for the closed singular cycles]\label{thm:functional-H}
For $n\in\mathbb N$ let $\widehat H_n:[0,1]\to H_1(M;\R)$ be the polygonal
interpolation of the points
\[
   \frac{H_k}{\sqrt n},\qquad k=0,1,\ldots,n.
\]
Then, for every $x\in M$, under $\Pp_x$,
\[
   \widehat H_n\Longrightarrow G_\cdot
   \qquad\text{in }C([0,1],H_1(M;\R)),
\]
where $G_\cdot$ is the Brownian motion of Theorem~\ref{thm:functional}.  The same convergence holds for Brownian motion with any initial distribution on $M$.

\end{theorem}

\begin{proof}
Fix $x\in M$.  Let $\widehat{\mathcal M}_n$ be the polygonal interpolation of
$\mathcal M_k/\sqrt n$.  By \eqref{eq:norm-bounded-distance},
\[
 \sup_{0\le s\le1}
 \|\widehat H_n(s)-\widehat{\mathcal M}_n(s)\|
 \le \frac{C}{\sqrt n}\longrightarrow0 .
\]
Compare $\widehat{\mathcal M}_n$ with the continuously sampled process
$n^{-1/2}\mathcal M_{n\cdot}$.  Fix $p>2$.  In harmonic coordinates the
quadratic variation accumulated on any interval of length at most one is
bounded uniformly.  By the Burkholder--Davis--Gundy inequality \cite[Chapter~IV]{RevuzYor}, the $p$-th moment of the maximum of a continuous martingale is bounded by a constant multiple of the $p/2$-th moment of its quadratic variation.  Since the latter is uniformly bounded on intervals of length one, there is a constant $C_p$ such that, uniformly in $k$,
\[
 \mathbf E_x\!
 \left[\sup_{0\le u\le1}
 \|\mathcal M_{k+u}-\mathcal M_k\|^p\right]\le C_p .
\]
For every $\varepsilon>0$,
\[
 \mathbf P_x\!\left(
 \max_{0\le k<n}\sup_{0\le u\le1}
 \|\mathcal M_{k+u}-\mathcal M_k\|>\varepsilon\sqrt n
 \right)
 \le \frac{C_p}{\varepsilon^p}n^{1-p/2}\longrightarrow0.
\]
On the interval $[k/n,(k+1)/n]$, the distance between the polygonal
interpolation and the continuously sampled martingale is at most twice
$\sup_{0\le u\le1}\|\mathcal M_{k+u}-\mathcal M_k\|$.  The preceding
estimate therefore shows that
\[
 \sup_{0\le s\le1}
 \left\|\widehat{\mathcal M}_n(s)-n^{-1/2}\mathcal M_{ns}\right\|
 \longrightarrow0
\]
in probability.  Slutsky's theorem, in the form saying that an $o_{\mathbf P}(1)$ perturbation does not change a weak limit, together with Theorem~\ref{thm:functional}, completes the proof; see \cite{Kallenberg}.
\end{proof}

\section{The linear normalization}

The linear normalization has an almost sure limit.

\begin{theorem}[Vanishing of the Schwartzman normalization]\label{thm:slln}
For every $x\in M$, under $\Pp_x$,
\[
   \frac{H_t}{t}\longrightarrow0
   \qquad\text{almost surely in }H_1(M;\R).
\]
The same conclusion holds for Brownian motion with any initial distribution on $M$.
\end{theorem}

\begin{proof}
Fix $x\in M$.  It is enough to pair with a basis of $H^1(M;\R)$.  Let $\omega$ be harmonic.  Since $M$ is compact,
\[
  \langle J^\omega\rangle_t
  =\int_0^t|\omega|^2(X_s)\,ds
  \le \|\omega\|_\infty^2 t.
\]
By the Dambis--Dubins--Schwarz theorem, a continuous local martingale can be represented as Brownian motion run at its quadratic-variation clock; see \cite[Chapter~V]{RevuzYor}.  Thus there is a one-dimensional Brownian motion $B$ such that, after the usual enlargement of the probability space if necessary,
\[
   J_t^\omega=B_{\langle J^\omega\rangle_t}.
\]
Since $\langle J^\omega\rangle_t\le \|\omega\|_\infty^2t$, the law of the iterated logarithm for Brownian motion \cite[Chapter~II]{RevuzYor} implies
\[
   \sup_{0\le u\le \|\omega\|_\infty^2t}|B_u|=o(t)
   \qquad\text{almost surely}.
\]
Hence $J_t^\omega/t\to0$ almost surely.  The closing term is uniformly bounded, hence also tends to zero after division by $t$.  Proposition \ref{prop:pairing} completes the proof.
\end{proof}

\begin{remark}
This theorem is the Brownian analogue of the zero Schwartzman cycle for a centered diffusive motion.  The nontrivial asymptotic information appears one order earlier, at the scale $\sqrt t$, and is described by Theorem \ref{thm:main}; Theorem \ref{thm:functional} gives the corresponding process-level statement for the harmonic homology martingale.
\end{remark}

\section{The circle and the planar winding problem}

For the circle the normalization and covariance can be computed explicitly.

\begin{example}[Brownian motion on the circle]\label{ex:circle}
Let $M=\Sph^1=\R/(2\pi\Z)$ with its standard metric and let $X_t$ be Brownian motion with generator $\frac12\frac{d^2}{d\theta^2}$.  If $\widetilde X_t$ is a lift to $\R$, then
\[
   \widetilde X_t=\widetilde X_0+B_t,
\]
where $B_t$ is standard one-dimensional Brownian motion.  Closing $X_{[0,t]}$ by an arc of length at most $\pi$ changes the lifted angular displacement by a bounded amount.  Thus the homology coordinate $H_t\in H_1(\Sph^1;\R)\simeq\R$ satisfies
\[
   H_t=\frac{B_t}{2\pi}+O(1)
\]
if the generator of $H_1(\Sph^1;\Z)$ is paired with $d\theta/(2\pi)$.  Consequently,
\[
   \frac{H_t}{\sqrt t}
   \Longrightarrow
   N\left(0,\frac{1}{4\pi^2}\right).
\]
This agrees with Theorem \ref{thm:main}, since
\[
   Q\left(\left[\frac{d\theta}{2\pi}\right],
          \left[\frac{d\theta}{2\pi}\right]\right)
   =\frac{1}{4\pi^2}.
\]
\end{example}

Spitzer's theorem concerns another process.  Let $Z_t$ be planar Brownian motion started away from the origin and write, continuously along the path,
\[
   Z_t=R_t e^{i\Theta_t}.
\]
The angular motion is a Brownian motion run at the random clock
\[
   A_t=\int_0^t\frac{ds}{R_s^2}.
\]
Here $R_t$ is a Bessel process of dimension $2$.  This clock is singular near the origin and is not asymptotically a deterministic multiple of $t$.  Spitzer's theorem states that
\[
   \frac{2\Theta_t}{\log t}
\]
converges in distribution to the standard Cauchy law \cite{Spitzer}.  Pitman and Yor developed a systematic study of these planar asymptotic laws and their random time changes \cite{PitmanYor}.

For Brownian homology on a closed manifold, the clock attached to a harmonic form is
\[
   \int_0^t|\omega|^2(X_s)\,ds.
\]
The function $|\omega|^2$ is bounded, and Proposition~\ref{prop:ergodic-start} gives
\[
   \frac1t\int_0^t|\omega|^2(X_s)\,ds
   \longrightarrow Q([\omega],[\omega])
   \qquad \Pp_x\text{-almost surely}
\]
for every $x\in M$.  The quadratic-variation clock is therefore asymptotically linear and deterministic, which gives the $\sqrt t$ normalization and the Gaussian law.

\section{Dependence on the metric and the covariance geometry}

The limiting law contains geometric information.  If $g$ is changed, then both Brownian motion and the harmonic representatives change.  The covariance form
\[
   Q_g(\alpha,\beta)
   =\frac{1}{\Vol_g(M)}
    \int_M\langle\omega_{\alpha,g},\omega_{\beta,g}\rangle_g\,dv_g
\]
is the Hodge inner product on $H^1(M;\R)$ normalized by total volume.  Thus the limiting law is the centered Gaussian measure on $H_1(M;\R)=H^1(M;\R)^*$ whose covariance form is this normalized Hodge inner product.

If $\omega_1,\ldots,\omega_{b_1}$ is any basis of harmonic $1$-forms and $\alpha_1,\ldots,\alpha_{b_1}$ the corresponding cohomology basis, the covariance matrix is
\begin{equation}\label{eq:Sigma}
   \Sigma_{ij}
   =\frac{1}{\Vol(M)}
    \int_M\langle\omega_i,\omega_j\rangle\,dv_g.
\end{equation}
The matrix is symmetric and positive definite.  Choosing a $Q$-orthonormal basis makes the limiting vector a standard Gaussian in $\R^{b_1}$.

\subsection{A flat torus}

Let $M=\R^n/\Lambda$ be a flat torus.  Choose linear coordinates on $\R^n$ and let $\ell_1,\ldots,\ell_n$ be a basis of linear forms whose classes determine a basis of $H^1(M;\R)$.  The harmonic representatives are the constant forms
\[
   \omega_i=d\ell_i.
\]
After fixing a lift of the initial point, Brownian motion on $M$ lifts to
\[
   \widetilde X_t=\widetilde X_0+B_t,
\]
where $B_t$ is Euclidean Brownian motion in $\R^n$.  Therefore
\[
   \int_0^t\omega_i\circ dX_s
   =\ell_i(B_t),
\]
and the corresponding coordinate of the closed homology class is
\[
   \langle[\omega_i],H_t\rangle
   =\ell_i(B_t)+R_t^i,
   \qquad |R_t^i|\le C_i.
\]
Thus
\[
   \frac{1}{\sqrt t}
   \bigl(\langle[\omega_1],H_t\rangle,\ldots,
   \langle[\omega_n],H_t\rangle\bigr)
   \Longrightarrow N(0,\Sigma),
\]
with
\[
   \Sigma_{ij}=\langle\omega_i,\omega_j\rangle_g.
\]
Since the forms are constant, this agrees with \eqref{eq:Sigma}.  In the rectangular case
\[
   M=\R^n/(L_1\Z\oplus\cdots\oplus L_n\Z),
\]
with the standard Euclidean metric and
\[
   \omega_i=\frac{dx_i}{L_i},
\]
then
\[
   \Sigma_{ij}=\frac{\delta_{ij}}{L_i^2}.
\]
Hence the limiting homology coordinates are independent centered Gaussians with variances $L_i^{-2}$.

\subsection{A closed hyperbolic surface}

Let $\Sigma_g$ be a closed oriented hyperbolic surface of genus $g\ge2$, with its metric of curvature $-1$.  Then
\[
   \dim H^1(\Sigma_g;\R)=2g.
\]
Choose a basis
\[
   \alpha_1,\ldots,\alpha_{2g}
\]
of $H^1(\Sigma_g;\R)$ and let
\[
   \omega_1,\ldots,\omega_{2g}
\]
be the harmonic representatives.  The covariance matrix of the limiting homology is
\begin{equation}\label{eq:hyperbolic-surface-cov}
   \Sigma_{ij}
   =\frac{1}{\operatorname{Area}(\Sigma_g)}
     \int_{\Sigma_g}\langle\omega_i,\omega_j\rangle\,dv.
\end{equation}
Since $\operatorname{Area}(\Sigma_g)=4\pi(g-1)$, this becomes
\[
   \Sigma_{ij}
   =\frac{1}{4\pi(g-1)}
     \int_{\Sigma_g}\langle\omega_i,\omega_j\rangle\,dv.
\]
Accordingly,
\[
   \frac{H_t}{\sqrt t}
   \Longrightarrow N(0,\Sigma)
   \qquad\text{in }H_1(\Sigma_g;\R).
\]
For an integral class $\alpha\in H^1(\Sigma_g;\Z)$, the harmonic representative determines a harmonic map
\[
   \phi_\alpha:\Sigma_g\to\Sph^1,
\]
and the coordinate $\langle\alpha,H_t\rangle$ is the lifted angular displacement of $\phi_\alpha(X_t)$ up to the bounded closing term.  Thus the covariance in \eqref{eq:hyperbolic-surface-cov} records the Hodge geometry of the hyperbolic surface rather than a constant Euclidean form.

\section{Stochastic holonomy}

The homology considered here is abelian, whereas stochastic parallel transport depends on the ordered path.  For a principal bundle with
compact structure group and a smooth connection, Brownian parallel transport
is defined by a horizontal Stratonovich equation \cite{Ito1963,Emery}.  In the
nonabelian case its holonomy is not determined by $H_t$; path ordering retains
information that disappears after passage to first homology.  For an abelian
flat connection, holonomy factors through homology, and the
results above apply after evaluating $H_t$ by the corresponding character.

\section*{Acknowledgments}

The first author would like to acknowledge Proyecto PAPIIT IN103324
(DGAPA, UNAM, M\'exico) for its financial support.  The authors also
acknowledge the use of ChatGPT (OpenAI) and Claude (Anthropic) for
proofreading, bibliographic searches, and refinement of the exposition.
The mathematical responsibility for the manuscript remains entirely
with the authors.

\end{document}